\documentclass[12pt, leqno]{amsart}
\usepackage{amsmath}
\usepackage{amsfonts}
\usepackage{amssymb}
\usepackage{amsthm}
\usepackage{url}
\usepackage{graphicx} 

\newcommand{\R}{\mathbb{R}}

\newcommand{\N}{\mathbb{N}}

\DeclareMathOperator*{\osc}{osc}

\DeclareMathOperator*{\dist}{dist}

\def\vint_#1{\mathchoice%
          {\mathop{\kern 0.2em\vrule width 0.6em height 0.69678ex depth -0.58065ex
                  \kern -0.8em \intop}\nolimits_{\kern -0.4em#1}}%
          {\mathop{\kern 0.1em\vrule width 0.5em height 0.69678ex depth -0.60387ex
                  \kern -0.6em \intop}\nolimits_{#1}}%
          {\mathop{\kern 0.1em\vrule width 0.5em height 0.69678ex depth -0.60387ex
                  \kern -0.6em \intop}\nolimits_{#1}}%
          {\mathop{\kern 0.1em\vrule width 0.5em height 0.69678ex depth -0.60387ex
                  \kern -0.6em \intop}\nolimits_{#1}}}

\newcommand{\art}[6]{{\sc #1, \rm #2, \it #3 \bf #4 \rm (#5), \mbox{#6}.}}

\newcommand{\AND}{{\rm and }}
\newcommand{\p}{{$p\mspace{1mu}$}}

\newcommand{\eps}{\varepsilon}

\theoremstyle{plain}
\newtheorem{theorem}[equation]{Theorem}
\newtheorem{lemma}[equation]{Lemma}

\numberwithin{equation}{section}

\theoremstyle{definition}

\theoremstyle{remark}
\newtheorem{remark}[equation]{Remark}
\title{Phragm\'en--Lindel\"of theorem for infinity harmonic
  functions}

\author{Seppo Granlund} \address[S.G.]{University of Helsinki,
  Department of Mathematics and Statistics, P.O. Box 68, FI-00014
  University of Helsinki, Finland} \email{seppo.granlund@pp.inet.fi}

\author{Niko Marola} \address[N.M.]{University of Helsinki, Department
  of Mathematics and Statistics, P.O. Box 68, FI-00014 University of
  Helsinki, Finland} \email{niko.marola@helsinki.fi}

\begin{document}

\keywords{Growth estimate, Infinity harmonic, \p-harmonic equation,
  Phragm\'en--Lindel\"of principle, \p-Laplace equation,
  \p-Laplacian.}

\subjclass[2010]{35B40, 35J70.}

\begin{abstract}
  We investigate a version of the Phragm\'en--Lindel\"of theorem for
  solutions of the equation $\Delta_\infty u=0$ in unbounded convex
  domains.  The method of proof is to consider this infinity harmonic
  equation as the limit of the \p-harmonic equation when $p$ tends to
  $\infty$.
\end{abstract}

\maketitle

\section{Introduction}

In this note we study the growth rate of infinity harmonic functions
without any restriction on the sign and, in particular, consider the
validity of the Phragm\'en--Lindel\"of type theorem for the infinity
harmonic equation
\begin{equation} \label{eq:Infty} \Delta_\infty u =
  \sum_{i,j=1}^n\frac{\partial u}{\partial x_i}\frac{\partial
    u}{\partial x_j}\frac{\partial^2 u}{\partial x_i\partial x_j}=0,
\end{equation}
in unbounded convex domains of $\R^n$, $n\geq 2$, with non-empty
boundary.  For more information and background on the equation
\eqref{eq:Infty} we refer to \cite{A}, \cite{ACJ}, and \cite{BDM}.
The classical Phragm\'en--Lindel\"of theorem \cite{PhLi} concerns the
growth of a subharmonic function $u$ in an unbounded sector in the
plane. Suppose that $D$ is a plane sector of angle $\pi/\alpha$ with
vertex at the origin. The theorem states that if $u$ is non-positive
on the boundary of $D$ and admits a positive value in $D$, then an
asymptotic growth condition
\[
\mathcal{M}(r) = \sup_{x\in D\cap B(0,r)}u(x) \gtrsim r^{\alpha}
\]
holds as $r$ tends to $\infty$.

Following Bhattacharya et al.~\cite{BDM} and Lindqvist--Manfredi~\cite{LiMa} 
our approach is to approximate the equation in \eqref{eq:Infty} by the \p-harmonic equation,
\begin{equation} \label{eq:pLaplace}
\nabla\cdot(|\nabla u|^{p-2}\nabla u)=0 \qquad
(1<p<\infty),
\end{equation}
and let $p$ grow to $\infty$.  

Several versions and extensions of Phragm\'en--Lindel\"of theorem to solutions of nonlinear elliptic equations and inequalities on unbounded domains of $\R^n$ can be found in the literature. To give a few references, we refer to \cite{CaVi}, \cite{JiLa}, \cite{Li}, and \cite{GLM} as well as to the references in these papers.

Our Phragm\'en--Lindel\"of type theorem for infinity harmonic functions, the main result in this note, is stated in Theorem~\ref{thm:PhLi}. Let us mention the related paper \cite{Bh}, see also \cite{Bh2},  where the author investigates non-negative solutions to \eqref{eq:Infty} on the half-space and on the exterior of the closed unit ball. It was shown that if a non-negative solution $u$ vanishes continuously on the hyperplane $\{x\in\R^n:\ x_n=0\}$, then either $u$ does not grow at all or $u$ is affine. However, the restriction on the sign is crucial in \cite{Bh} for obtaining the estimates on growth rates.

\medskip

{\bf Acknowledgements:} Work partially done during the second author's visit to the Institut Mittag-Leffler (Djursholm, Sweden). Support by the Institut Mittag-Leffler and the V\"ais\"al\"a Foundation are gratefully acknowledged. The authors would like to thank Peter Lindqvist for his comments, Teemu Lukkari for pointing out some references, and the referee for useful suggestions.

\section{Phragm\'en--Lindel\"of principle}
\label{sect:PhLi}

Let $D$ be a unbounded convex domain of $\R^n$, $n\geq 2$, with 
non-empty boundary $\partial D$. An open ball centered at $x$ and of
radius $r$ is written as $B(x,r)$ or, for brevity, as $B_r$, and we
write the closure of a ball as $\overline{B}(x,r)$ or
$\overline{B}_r$. The oscillation $\sup_{B(x,r)}u-\inf_{B(x,r)}u$ of a
function $u$ on $B(x,r)$ is written as $\osc(u; B(x,r))$.

We interpret a non-divergence form equation \eqref{eq:Infty} in the
viscosity sense.  An upper semicontinuous function $u$ is said to be
$\infty$-subharmonic in $D$ if, for every $x_0\in D$ and every
$\varphi\in C^2(D)$ such that $u-\varphi$ has a local maximum at $x_0$, we have $\Delta_\infty \varphi(x_0) \geq 0$. A lower semicontinuous function $u$ is said
to be $\infty$-superharmonic in $D$ if $-u$ is $\infty$-subharmonic in
$D$. A function $u\in C(D)$ is said to be $\infty$-harmonic in $D$ if
it is both $\infty$-subharmonic and $\infty$-superharmonic in $D$.

In what follows, $\phi:\R\to\R$ shall be a convex function satisfying
the following conditions: there exists a sub-interval $I$ of $\R$ such
that
\begin{enumerate}

\item[(C-1)] $\quad \phi\in C^2(I)$, \smallskip

\item[(C-2)] $\quad \phi'(t)^2 \leq \phi''(t) $ for all $t\in I$, \smallskip

\item[(C-3)] $\quad \phi'(t)>0$ for
  all $t\in I$, \smallskip
  
\item[(C-4)] $\quad \sup_{t\in I}\phi'(t) \leq M$ for some positive constant $M$.
\end{enumerate}

The following estimate is straightforward to verify but for the sake
of completeness we shall provide a proof. Our proof is rather similar to the argument in the proof of the main theorem in \cite{LiMa}. In \cite{LiMa} a 
logarithmic function is considered instead of a general convex 
function $\phi$.

\begin{lemma} \label{lemma:1} Let $u$ be a $\infty$-harmonic function
  in $D$. Suppose that the conditions (C-1) through (C-4) are valid for a
  convex function $\phi:I\to\R$. Then for every $\overline{B}_r\subset
  D$ and $0<\delta<1$ we have the estimate
\begin{equation} \label{eq:lemma12}
\|\nabla \phi(u)\|_{L^\infty(B_{(1-\delta)r})} \leq \frac1{\delta r}.
\end{equation}
\end{lemma}

\begin{proof}
Let us suppose that $u_p$ is \p-harmonic in $D$. Then for every
  $\overline{B}_r\subset D$ we have the inequality
  \begin{equation} \label{eq:lemma1} 
  \left(\int_{B_r}|\nabla
    \phi(u_p(x))|^p\xi^p\, dx\right)^{1/p} \leq
    \frac{p}{p-1}\left(\int_{B_r}|\nabla \xi|^p\,dx\right)^{1/p},
\end{equation}
whenever $\xi\in C_0^\infty(B_r)$ is non-negative. The inequality in
\eqref{eq:lemma1} can be easily verified by testing with the function
$\eta(x) = \phi'(u_p(x))^{p-1}\xi^p(x)$ in the weak formulation of the
\p-Laplace equation \eqref{eq:pLaplace} and using the conditions (C-1)--(C-3) for the function $\phi$ as well as the H\"older inequality:
\begin{align*}
(p-1) & \int_{B_r}\phi'(u_p(x))^{p-2}\phi''(u_p(x))\xi^p|\nabla u_p|^p\, dx \\
& \quad \leq
p\int_{B_r}\phi'(u_p(x))^{p-1}|\nabla u_p|^{p-1}\xi^{p-1}|\nabla\xi|\, dx \\
& \quad \leq p\left(\int_{B_r}\phi'(u_p(x))^p\xi^p|\nabla u_p|^p\, dx\right)^{(p-1)/p}\left(\int_{B_r}|\nabla\xi|^p\, dx\right)^{1/p}.
\end{align*}
For the proof of Theorem~\ref{thm:PhLi} it shall be important to know that the 
inequality in \eqref{eq:lemma1} holds 
also for \p-subsolutions of the \p-Laplace equation \eqref{eq:pLaplace} by an analogous argument. 
  
By the results in \cite{BDM} and \cite{J}, and reasoning as in \cite{LiMa}, we can conclude the following: there exists a sequence $p_j$ of $p_j$-harmonic functions $u_{p_j}$ such that $u_{p_j}\to u$ in $C^{\alpha}(\overline{B}_r)$ for any $\alpha \in [0,1)$, and $u_{p_j}\rightharpoonup u$ in $W^{1,m}(B_r)$
for any finite $m$ as $p_j\to\infty$. 

Fix $m\geq n$. Then for any $p>m$, we have
\begin{align*}
& \left(\int_{B_r}|\nabla \phi(u_p(x))|^m\xi^m\, dx\right)^{1/m} \\
& \qquad \leq \left(\int_{B_r}|\nabla \phi(u_p(x))|^p\xi^p\, dx\right)^{1/p}|B_r|^{1/m-1/p} \\
  & \qquad \leq \frac{p}{p-1}\left(\int_{B_r}|\nabla
    \xi|^p\,dx\right)^{1/p}|B_r|^{1/m-1/p}.
\end{align*}
By the aforementioned convergence result and the condition (C-4) for the function $\phi$ it is straightforward to check that $\nabla\phi(u_p)$ converges weakly in $L^m(B_r)$ to $\nabla\phi(u)$ as $p\to\infty$.

In conclusion, we may justify, by first letting $p\to\infty$ and then $m\to\infty$, that we obtain the inequality in \eqref{eq:lemma12} by choosing $\xi$ to be a suitable radial test function. 
\end{proof}

The following result on the growth of infinity harmonic functions is the main result of this note. We remind the reader that there is no restriction on the sign of an $\infty$-harmonic function $u$ in $D$.

\begin{theorem} \label{thm:PhLi} Let $D$ be a unbounded convex 
domain in $\R^n$ with non-empty boundary $\partial D$.  
Suppose that $u\in C(\overline{D})$ is an $\infty$-harmonic function in $D$ such that $u|_{\partial D}=0$ and $u$ admits a positive value in $D$.
Let $x_0\in \partial D$ and let
\[
\kappa_0:=\sup_{r\in\R_+}\frac{|D\cap B(x_0,r)|}{|B(x_0,r)|}.
\]
Define 
\[
\mathcal{M}(r) = \sup_{x\in D\cap B(x_0,r)}u(x).
\]
Then
\begin{equation} \label{eq:PL}
\liminf_{r\to\infty}\frac{\log\mathcal{M}(r)}{\log r} \geq \alpha>0,
\end{equation}
where the number $\alpha$ depends only on $n$ and $\kappa_0$, and has the expression
\[
\alpha = -\log_4\left(1-e^{-C\kappa_0^{1/n}}\right).
\]
Here $C$ is a positive constant depending only on $n$. In particular, $\alpha \to \infty$ as $\kappa_0\to 0$.
\end{theorem}

\begin{proof}
 Our goal is to prove the following oscillation inequality 
\[
\mathcal{M}(r) \leq \theta\mathcal{M}(4r),
\] 
from which the assertion follows by iterating. The constant $0<\theta<1$ shall depend only on $n$ and $\kappa_0$.  
  
Let $x_0\in\partial D$ be fixed. We may apply a result in \cite{BDM} since the boundary of the set $D\cap B(x_0,R)$, $R>0$, satisfies the cone property. Hence, by \cite{BDM} there exists a sequence $p_j$ of $p_j$-harmonic functions $u_{p_j}$ such that $u_{p_j}|_{\partial D}=u|_{\partial D}=0$ and $u_{p_j}\to u$ in
$C^{\alpha}(\overline{D}\cap \overline{B}(x_0,R))$ for any $\alpha
  \in [0,1)$ and $u_{p_j}\rightharpoonup u$ in $W^{1,m}(D\cap
  B(x_0,R))$ for any finite $m$ as $p_j\to\infty$. 

Up to the pointwise estimate in \eqref{eq:GradPointwise} the main part of our proof is analogous to the proof of Lemma~\ref{lemma:1}.  We next consider the sequence of cut-off functions $\max\{u_{p_j},0\}$ which still converges to $\max\{u,0\}$ uniformly in $C^{\alpha}(\overline{D}\cap \overline{B}(x_0,R))$ and $\max\{u_{p_j},0\}$ 
is a non-negative $p_j$-subsolution to the equation in \eqref{eq:pLaplace}. The a priori estimates in the proof of Lemma~\ref{lemma:1} hold for the functions in the sequence. 

Let $\delta>0$ be small enough. The function $\max\{u_{p_j}(x)-\delta,0\}$ is a $p_j$-subsolution to the equation in \eqref{eq:pLaplace} in $D\cap B(x_0,4r)$. We define the functions
\[
h_{p_j}(x) = \left\{\begin{array}{ll}\max\{u_{p_j}(x),0\} & \textrm{if }\, x\in D\cap B(x_0,4r), \\
0 & \textrm{if }\, x\in B(x_0,4r)\setminus D, \end{array}\right.
\]
and 
\[
  h_{p_j}^\delta (x) = \left\{\begin{array}{ll}\max\{u_{p_j}(x)-\delta,0\} & \textrm{if }\, x\in D\cap B(x_0,4r), \\
      0 & \textrm{if }\, x\in B(x_0,4r)\setminus D, \end{array}\right.
\] 
from which the latter is a non-negative $p_j$-subsolution to \eqref{eq:pLaplace} in the ball $B(x_0,4r)$. Then well known a priori estimates for non-negative subsolutions imply that there exists a subsequence such that $h_{p_j}^\delta$ converges to $h_{p_j}$ in $W^{1,p}(B(x_0,4r))$ and, moreover, that $h_{p_j}$ is a non-negative $p_j$-subsolution in $B(x_0,4r)$.

Let us now fix any $\eps>0$ and define the function
  \[
  \phi(t) =
  -\log\left(\frac{\mathcal{M}(4r)-t+\eps}{\mathcal{M}(4r)+\eps}\right)
  \]
  for $t\in(-\infty,\mathcal{M}(4r)]$. The function $\phi$ is convex,
  $\phi(0)=0$, satisfies the conditions (C-1)--(C-3), and the condition
   (C-4) is valid since $\sup \phi'(t) \leq 1/\eps$.

We consider the non-negative $p_j$-subsolution $h_ {p_j}$ to the equation in \eqref{eq:pLaplace} on $B(x_0,4r)$. By letting $p_j\to\infty$ and proceeding as in the proof of Lemma~\ref{lemma:1} we have the following pointwise estimate for
almost any $x\in D\cap B(x_0,2r)$, and eventually for almost every $x\in B(x_0,2r)$ since the function $h$ defined below is constant in $B(x_0,4r)\setminus D$,
\begin{equation} \label{eq:GradPointwise}
|\nabla \phi(h(x))| \leq \frac2{r},
\end{equation}
where 
\[
h(x) = \left\{\begin{array}{ll}\max\{u(x),0\} & \textrm{if }\, x\in D\cap B(x_0,4r), \\
0 & \textrm{if }\, x\in B(x_0,4r)\setminus D. \end{array}\right.
\]
The function $h$ is $\infty$-subharmonic in $D\cap B(x_0,4r)$ but not necessarily $\infty$-harmonic as $u$ need not be positive in $D$ (cf. Lemma~\ref{lemma:1}).

Integrating the estimate in \eqref{eq:GradPointwise} over the ball $B(x_0,2r)$
we obtain 
\begin{align*}
\int_{B(x_0,2r)}|\nabla \phi(h(x))|^n\, dx & = \int_{D\cap B(x_0,2r)}|\nabla \phi(h(x))|^n\, dx \\
& \leq C\frac{|D\cap B(x_0,2r)|}{|B(x_0,2r)|},
\end{align*}
where $C$ depends only on $n$.

The $\infty$-harmonic function $u$ is known to be monotone in the
sense of Lebesgue \cite{J}. Since the composite function $\phi\circ h$
can also be seen to be monotone we have 
\[
\osc(\phi(h); \partial B_r) =
\osc(\phi(h); B_r). 
\]
Hence, by applying the well-known
Gehring--Mostow oscillation lemma for monotone functions we obtain the inequality
\begin{align*}
  \left(\osc(\phi(h); B(x_0,r))\right)^n \log \frac{2r}{r} & \leq
  C\int_{B(x_0,2r)}|\nabla \phi(h(x))|^n\, dx,
\end{align*}
where $C$ is a positive constant depending on $n$. In conclusion, since the composite function $\phi(h)$ is non-negative and vanishes at some point, we
obtain
\begin{align*}
  \sup_{D\cap B(x_0,r)}\phi(h) & = \osc(\phi(h); B(x_0,r)) \\
  & \leq C \left(\frac{|D\cap B(x_0,2r)|}{|B(x_0,2r)|}\right)^{1/n}
  \leq C\kappa_0^{1/n}.
\end{align*}
It follows that
\[
\frac{\mathcal{M}(4r)-\mathcal{M}(r)+\eps}{\mathcal{M}(4r)+\eps} \geq e^{-C\kappa_0^{1/n}},
\]
where $C$ is a positive constant depending on $n$ only. Letting
$\eps\searrow 0$, we obtain the desired oscillation inequality
\[
\mathcal{M}(r) \leq \left(1-e^{-C\kappa_0^{1/n}}\right)\mathcal{M}(4r) =
\theta\mathcal{M}(4r).
\]
We iterate the preceding inequality to get for any $\nu\in\N$
\[
\mathcal{M}(4^\nu r) \geq \left(\frac{4^\nu r}{r}\right)^{-\log_4 \theta}\mathcal{M}(r),
\]
from which the asymptotic behavior in the assertion in the theorem follows with
\[
\alpha = -\log_4\left(1-e^{-C\kappa_0^{1/n}}\right)>0.
\]
\end{proof}

On a related note, we may derive the analogue of
Theorem~\ref{thm:PhLi} for a \p-harmonic function $u$ in $D$ for each
$1<p<\infty$. In case $p\geq n$ either the Gehring--Mostow or the
Morrey oscillation estimate can be utilized similarly as in this
paper. When $1<p<n$, one needs $L^\infty-L^p$ estimates of De
Giorgi--Ladyzhenskaya--Ural'tseva-type for $\phi(u)$, where $\phi$ is
a convex function chosen as in the proof of Theorem~\ref{thm:PhLi}.

\begin{remark}
\begin{itemize}
\item[] 
\item[(1)] The convexity condition on $D$ in Theorem~\ref{thm:PhLi} can be relaxed to allow domains of more general form. For instance, one can consider a general 
domain $D$ such that $D\cap B(x_0,R)$, where $R\geq R_0>0$, is a domain 
and its boundary satisfies the cone property.  \smallskip
\item[(2)] The condition $u|_{\partial D}=0$ can be replaced by the weaker condition
\[
\limsup_{\substack{x\to x_0 \\ x_0\in \partial D}}u(x) \leq 0.
\]\smallskip
\item[(3)] In the preceding proof, an oscillation lemma due to Gehring
  and Mostow can be replaced with Morrey's lemma for $p>n$. In this
  case monotonicity is not used. Moreover, the estimate in
  Lemma~\ref{lemma:1} suggests a possibility to avoid both of these
  oscillation estimates and instead work directly at the
  $L^\infty$-level. We did not pursue this approach here.
\end{itemize}
\end{remark}


\begin{thebibliography}{99}

\bibitem{A} \art{Aronsson, G.}  {On the partial differential equation
    $u_{x}^{2}u_{xx} +2u_{x}u_{y}u_{xy}+u_{y}^{2}u_{yy}=0$}
  {Ark. Mat.}{7}{1968}{395--425}

\bibitem{ACJ} \art{Aronsson, G., Crandall, M. G., \AND Juutinen, P.}
{A tour of the theory of absolutely minimizing functions}
{Bull. Amer. Math. Soc. (N.S.)}{41}{2004}{439--505}

\bibitem{Bh} \art{Bhattacharya, T.}
{On the behaviour of $\infty$-harmonic functions on some special unbounded domains} {Pacific J. Math.}{219}{2005}{237--253} 

\bibitem{Bh2} \art{Bhattacharya, T.}
{A note on non-negative singular infinity-harmonic functions in the half-space} 
{Rev. Mat. Complut.}{18}{2005}{377--385}

\bibitem{BDM} \art{Bhattacharya, T., DiBenedetto, E., \AND Manfredi, J.}
	{Limits as $p\to\infty$ of $\Delta_pu_p=f$ and related extremal problems}
 {Rend. Sem. Mat. Univ. Politec. Torino}{}{1989}{15--68}

\bibitem{CaVi} \art{Capuzzo Dolcetta, I. \AND Vitolo, A.}
{A qualitative Phragm\'en--Lindel\"of theorem for fully nonlinear elliptic equations}{J. Differential Equations}{243}{2007}{578--592}

\bibitem{GLM} \art{Granlund, S., Lindqvist, P., \AND Martio, O.}
{Phragm\'en-Lindel\"of's and Lindel\"of's theorems} 
{Ark. Mat.}{23}{1985}{103--128}

\bibitem{J} \art{Jensen, R.}
  {Uniqueness of Lipschitz extensions: minimizing the sup norm of the gradient}
  {Arch. Rational Mech. Anal.}{123}{1993}{51--74}

\bibitem{JiLa} \art{Jin, Z. \AND Lancaster, K.}
{Theorems of Phragm\'en--Lindel\"of type for quasilinear elliptic equations}{ J. Reine Angew. Math.}{514}{1999}{165--197}

\bibitem{Li} \art{Lindqvist, P.}{On the growth of the solutions of the differential equation ${\rm div}(\vert\nabla u\vert^{p-2}\nabla u)=0$ in $n$-dimensional space}{J. Differential Equations}{58}{1985}{307--317}

\bibitem{LiMa} \art{Lindqvist, P. \AND Manfredi, J.}
{The Harnack inequality for $\infty$-harmonic functions} 
{Electron. J. Differential Equations}{4}{1995}{1--5}

\bibitem{PhLi} \art{Phragm\'en, E. \AND Lindel\"of, E.}{Sur une
    extension d'un principe classique de l'analyse et sur quelques
    propri\'et\'es des fonctions monog\`enes dans le voisinage d'un
    point singulier}{Acta Math.}{31}{1908}{381--406}

\end{thebibliography}
\end{document}